\documentclass{article}

\author{Ian Benway}
\title{On Isolated Geometric Triangulations}
\date{}

\usepackage{graphicx,xypic}
\usepackage{amsthm}
\usepackage{amsmath,amssymb}
\usepackage{amsfonts}
\usepackage{tikz}
\usepackage[shortlabels]{enumitem}
\usepackage{algpseudocode}

\usepackage{caption}
\usepackage{subcaption}
\usepackage{float}
\usepackage{url}
\usepackage{svg}
\usepackage{geometry}[margin=1in]
\usepackage{hyperref}
\usepackage{amsthm}
\usepackage{changepage} % for adjustwidth

\newtheorem{theorem}{Theorem}[section]
\newtheorem{lemma}[theorem]{Lemma}
\newtheorem{proposition}[theorem]{Proposition}

\newtheorem{corollary}[theorem]{Corollary}

\newcommand{\HH}{\mathbb H}
\newcommand{\ZZ}{\mathbb Z}
\newcommand{\SL}{\text{SL}}

\newcommand{\CC}{\mathbb C}
\renewcommand{\HH}{\mathbb H}

\usepackage{blindtext}

\begin{document}
\maketitle

%%% Contents %%%
\begin{abstract}
    Work of Kalelkar, Schleimer, and Segerman shows that, with some exceptions, the set of essential ideal triangulations of an orientable cusped hyperbolic 3-manifold is connected via 2-3 and 3-2 moves. It is natural to ask if the subgraph consisting of only those triangulations that are geometric is connected. Hoffman gives the first two examples of geometric triangulations that are isolated, meaning that no 2-3 or 3-2 move results in a geometric triangulation. In this paper, we exhibit an infinite family of once-punctured torus bundles whose canonical geometric triangulation is isolated.
\end{abstract}
\section{Introduction}
Let $M$ be an orientable cusped hyperbolic 3-manifold. We consider the set $\mathcal T(M)$ of all ideal triangulations of $M$. Given a triangulation $T$ with more than one tetrahedron, we can obtain a new triangulation by applying a 2-3 move; see Figure \ref{fig:2-3 moves two views}. We can then consider the \textbf{bistellar flip graph}, denoted $\mathbb T(M)$, whose node set is $\mathcal T(M)$ and for which an arc exists between two triangulations if and only if there is a 2-3 move bringing one to the other. A theorem of Amendola shows that $\mathbb T(M)$ is connected; see \cite{Amendola_2005}, \cite{matveev2007algorithmic}, \cite{article}, and \cite{rubinstein2019traversingthreemanifoldtriangulationsspines}.

We can ask for our triangulations to satisfy nicer properties, and ask if the corresponding induced subgraphs of $\mathbb T(M)$ are likewise connected. For example, Segerman shows that the subgraph of $\mathbb T(M)$ induced by triangulations having no degree-1 edge is connected in the case where $T$ has one vertex \cite{segerman2016connectivitytriangulationsdegreeedges}. Call a triangulation \textbf{essential} if it admits solutions to the gluing equations in $\CC\setminus\{0,1\}$. Define $\mathbb T^\circ_E(M)$ to be the subgraph of $\mathbb T(M)$ induced by essential triangulations which admit some 2-3 or 3-2 move that preserves essentiality\footnote{There are essential triangulations which are \textit{isolated}: the triangulation is essential but no 2-3 or 3-2 move produces an essential triangulation. We expand on this in Proposition \ref{prop: isolated not geom} and justify ignoring isolated essential triangulations for the purposes of this paper.}. Kalelkar, Schleimer, and Segerman show that $\mathbb{T}^\circ_E(M)$ is connected \cite{kalelkar2024connectingessentialtriangulationsii}. %Motivation here might be nice: upshot is properties invariant under 2-3 and 3-2 moves are invariants of the manifold
Work of Thurston shows that such solutions which lie in $\mathbb H=\{z\in\mathbb C\:|\:\text{Im}(z) >0\}$ correspond to positively-oriented tetrahedra which induce a hyperbolic structure of strictly positive volume on $M$. If the triangulation $T$ admits a solution to the gluing equations in $\mathbb H$ which also satisfies certain completeness equations, then the tetrahedra glue together to give a smooth manifold with a complete hyperbolic structure, and we call $T$ a \textbf{geometric triangulation} \cite{thurston1979geometry}.
% It's then natural to ask whether the subgraph $\mathbb T_G(M)$ of $\mathbb T_E(M)$ consisting of geometric triangulations of $M$ is connected under 2-3 and 3-2 moves. 

While it has long been known that every finite-volume cusped hyperbolic 3-manifold has a topological ideal triangulation \cite{bing3manifoldtriangulations}, it is still an open question as to whether every finite-volume cusped hyperbolic 3-manifold admits a geometric triangulation. It then becomes of interest to understand the set of geometric triangulations of a fixed manifold. For example, Dadd and Duan show that the figure-8 knot complement has infinitely many geometric triangulations \cite{dadd2015constructinginfinitelygeometrictriangulations}. Futer, Hamilton, and Hoffman show that this property holds virtually -- every cusped hyperbolic 3-manifold has a finite cover which admits infinitely many geometric triangulations \cite{Futer_2022}. In this paper, we investigate the set of geometric triangulations by means of the induced subgraph $\mathbb T_G(M)$ of $\mathbb T_E^\circ(M)$ consisting of geometric triangulations. We call this the \textbf{geometric bistellar flip graph}. Given the context above, it is natural to ask whether $\mathbb T_G(M)$ is connected.

It seems the only discussion of this in the literature is two counterexamples by Neil Hoffman in the figure-8 knot complement; see Remark 3.3 in \cite{dadd2015constructinginfinitelygeometrictriangulations}. These two geometric triangulations, obtained by the isomorphism signatures\footnote{The software package Regina can turn these isomorphism signatures into triangulations \cite{regina}, \cite[Section~3.2]{Burton_2011}.} \textit{fLLQcacdedejbqqww} and \textit{fLLQccecddehqrwwn}, share the property that any possible 2-3 or 3-2 move results in a triangulation which is no longer geometric. We call a triangulation with this property an \textbf{isolated geometric triangulation}.

Beyond these two small examples, not much is known about isolated geometric triangulations. In this paper, we expand the set of such examples in the following theorem:\\\\
\textbf{Theorem \ref{theorem: main}.} \textit{The once-punctured torus bundle associated to the cyclic word $R^{N}L^{M}$ has an isolated geometric triangulation for all $N$ and $M$ which are even and sufficiently large.}\\\\
In particular, this theorem shows that having an isolated geometric triangulation is not a property of small manifolds. We further show that for each of these triangulations of once-punctured torus bundles, there is a sequence of 2-3 and 3-2 moves which passes through triangulations containing a flat tetrahedron and results in a new geometric triangulation. From this, we conclude the following:\\\\
\textbf{Corollary \ref{cor:infinitely many disconnected}} \textit{There are infinitely many cusped hyperbolic 3-manifolds for which the geometric bistellar flip graph is disconnected.}

\begin{figure}
     \centering
     \captionsetup{width=.8\linewidth}
     \includegraphics[width=0.9\linewidth]{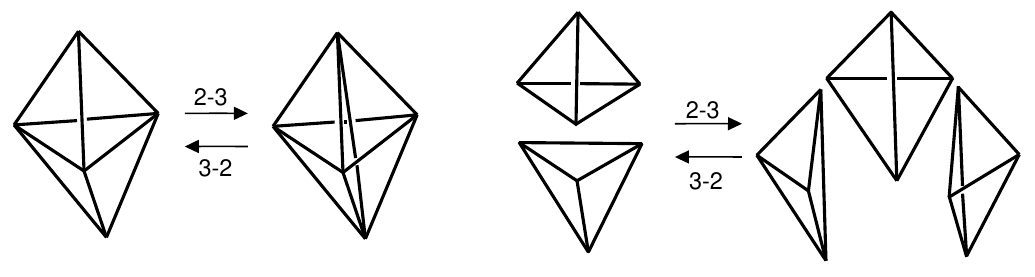}
     \hfill
        \caption{Two views of the 2-3 move}
        \label{fig:2-3 moves two views}
\end{figure}

\subsection{Acknowledgments}
The author would like to thank Henry Segerman and Alex He for their incredible support, comments, and advice. Thank you to the referee for insightful comments on many parts of this paper. Thanks to Isaiah DeHoyos for divulging some useful trigonometry tricks. The author was supported in part by
National Science Foundation grant DMS-2203993.
\section{Background}

\subsection{Geometric Triangulations and Local Moves}\label{subsection: geometric triangulations and local moves}
For background on geometric triangulations, we refer the reader to Chapter 4 of \cite{purcell2020hyperbolicknottheory}. Let $T$ be an ideal triangulation of an orientable cusped hyperbolic 3-manifold $M$. A \textbf{2-3 move}\footnote{This is an example of a bistellar flip and is sometimes called a Pachner move.} is a local modification of the triangulation which replaces a bipyramid consisting of two distinct tetrahedra which share a common face with a union of three distinct tetrahedra which share a common edge. This is illustrated in Figure \ref{fig:2-3 moves two views}. This replaces one 3-ball with another, so a 2-3 move changes the combinatorics of $T$ but not its topology. The inverse of a 2-3 move is called a 3-2 move, and can be applied at a degree-3 edge whenever the three tetrahedra around that edge are all distinct. As mentioned in the introduction, 2-3 and 3-2 moves are sufficient to get from one triangulation of $M$ to any other\footnote{Here we use the assumption that $M$ is orientable: there are no ideal triangulations of orientable cusped hyperbolic 3-manifolds with only one tetrahedron. This is not true, however, for the non-orientable Gieseking manifold.}; see \cite[Theorem 1.2.5]{matveev2007algorithmic} and \cite{Amendola_2005}. Moreover, 2-3 and 3-2 moves are sufficient to get between any two essential triangulations (triangulations which admit solutions in $\mathbb C\setminus\{0,1\}$ to the gluing equations) while only passing through essential triangulations, if one ignores those triangulations which do not admit any 2-3 or 3-2 moves that preserve essentiality \cite[Theorem 3.5]{kalelkar2024connectingessentialtriangulationsii}. Such triangulations are called \textbf{isolated essential triangulations}, and the following proposition justifies ignoring them for the purposes of this paper.
\begin{figure}
    \centering
    \captionsetup{width=.8\linewidth}
\includegraphics[width=0.7\linewidth]{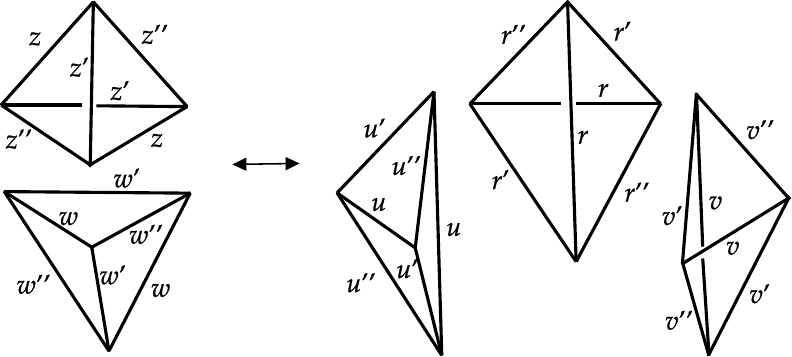}
    \caption{A 2-3 move labeled with shape parameters.}
    \label{fig:2-3 with shapes}

\end{figure}
\begin{proposition}
\label{prop: isolated not geom}
    If $T$ is an isolated essential triangulation, then $T$ is not geometric.
\end{proposition}
\begin{proof}
    If $T$ is geometric, then each tetrahedron has a shape parameter which is a solution to the gluing equations with positive imaginary part. Because such solutions are unique \cite[Section 6.1]{purcell2020hyperbolicknottheory}, we can track the shape parameters to the new tetrahedra whenever we perform a 2-3 or 3-2 move. We will deal with the case where we do a 2-3 move (the case where we do a 3-2 move is nearly identical). Let $z$ and $w$ be the shape parameters of the bipyramid before the 2-3 move and $r$, $u$, and $v$ be the shape parameters of the three tetrahedra after the 2-3 move as in Figure $\ref{fig:2-3 with shapes}$. Let $z' = \frac{1}{1-z}$ and $z''=\frac{z-1}{z}$ and likewise for $w$. We have $r = z'w'$, $u=wz''$, and $v=zw''$. Since $T$ is geometric, $z$, $z'$, $z''$, $w$, $w'$, and $w''$ all have positive imaginary part.  Thus $r$, $u$, and $v$ cannot be $1$, $0$, or $\infty$. So the result of the 2-3 move is essential.
\end{proof}

In other words, $\mathbb T_G(M)$ is contained in the main connected component of the essential bistellar flip graph $\mathbb{T}^\circ_E(M)$. We now investigate the connectedness of $\mathbb T_G(M)$ itself.

Suppose two triangulations $T$ and $T'$ differ by a 2-3 move. If both $T$ and $T'$ are geometric triangulations, we will say they differ by a \textbf{geometric 2-3 move}. In this case, the edge introduced by the 2-3 move can be taken to be a geodesic, and the tetrahedra of $T'$ are all positively-oriented. 

It is useful to understand the obstructions which may cause a 2-3 move to be not geometric. As a theme throughout this paper, we will take advantage of the Euclidean structure of the cusp to understand obstructions in the hyperbolic structure of the 3-manifold. Recall that a geometric triangulation of $M$ induces a Euclidean triangulation of the torus boundary, which we call the cusp triangulation. Figure \ref{fig:2-3 ideal cusp} illustrates the effect of a 2-3 move on the cusp triangulation. In particular, a 2-3 move induces three 2-2 moves and two 1-3 moves on the cusp, as shown on the right of Figure \ref{fig:2-3 ideal cusp}. A \textbf{geometric 2-2 move} is a 2-2 move such that each involved edge can be taken to be a Euclidean geodesic; see Figure \ref{figs: 2-2 moves}. Note that a geometric 2-3 move must induce three geometric 2-2 moves on the cusp. For each cusp of $M$, the cusp triangulation induced by $T$ is unique up to similarity, thus an obstruction to these geometric 2-2 moves induces an obstruction to the corresponding 2-3 move in $T$. In other words, to show that a 2-3 move is not geometric, as we will in the proof of Theorem \ref{theorem: main}, we can show that the corresponding 2-2 moves on the cusp are not geometric.

\begin{figure}
     \centering
     \captionsetup{width=.8\linewidth}
     \includegraphics[width=\linewidth]{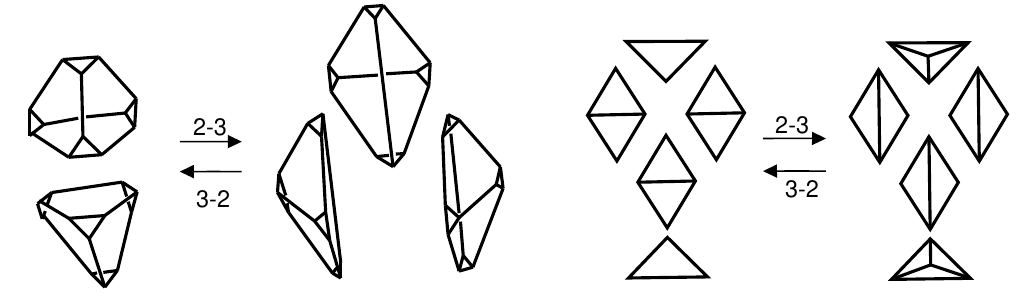}
     \hfill
        \caption{Left: The 2-3 move applied to ideal tetrahedra. Right: The effect of the 2-3 move on the cusp triangulation.}
        \label{fig:2-3 ideal cusp}
\end{figure}
\begin{figure}
    \centering
    \captionsetup{width=.8\linewidth}
    \includegraphics[width=\linewidth]{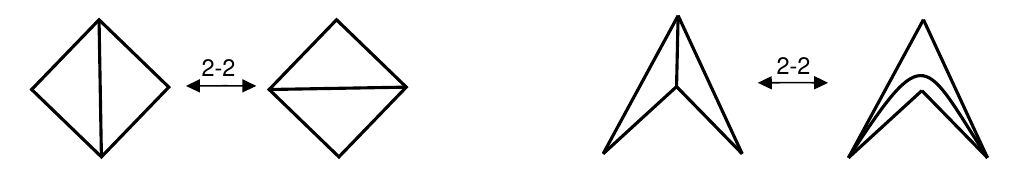}
    \caption{Left: A (geometric) 2-2 move. \textit{Right:} A 2-2 move, the result of which cannot be geometric. In other words, either the introduced edge is not a geodesic (as pictured) or a triangle becomes negatively-oriented.}
    \label{figs: 2-2 moves}
\end{figure}
%Talk about local moves, introduce pseudogeometric triangulations, obstructions of geometry at local moves, blah blah blah
%Picture of 2-3 effect on cusp (three 2-2 moves)
\subsection{Once-Punctured Torus Bundles}
We now define the \textbf{monodromy ideal triangulation} of a once-punctured torus bundle. This is a particularly nice and natural way to triangulate the manifold. Using techniques of Casson and Rivin, Gu\'eritaud shows that this triangulation is geometric \cite{Gu_ritaud_2006}. Lackenby shows that it's isomorphic to its canonical (Epstein-Penner) decomposition \cite{lackenby2001canonicaldecompositiononcepuncturedtorus}. Jaco, Rubinstein, Spreer, and Tillmann show that it's also isomorphic to a minimal triangulation \cite{Jaco_2019}.

For the construction of the monodromy ideal triangulation, we follow work of Floyd and Hatcher \cite{FLOYD1982263} and Gu\'eritaud \cite{Gu_ritaud_2006}. Let $\varphi \in \SL_2(\ZZ)$ be a matrix with two distinct real eigenvalues. Then $\varphi$ determines a homeomorphism of the torus fixing the ``origin" $0\in T^2$, so that we may define the 3-manifold
$$ M_\varphi = ((T^2 -\{0\})\times I )/\sim $$
where $(x,0) \sim (\varphi(x), 1)$. The map $\varphi$ is the monodromy associated to $M_\varphi$. By Thurston's hyperbolization theorem \cite{thurston1979geometry}, $M_\varphi$ admits a finite-volume hyperbolic structure.

Let
\[ L = \begin{pmatrix}
    1&0\\1&1
\end{pmatrix}\;\;\;\;\;\;\;\;\;\;\text{ and } \;\;\;\;\;\;\;\;R = \begin{pmatrix}
    1&1\\0&1
\end{pmatrix}. \]
Then the conjugacy class of $\varphi$ contains an element of the form
\[ V\varphi V^{-1} = R^{a_1}L^{b_1}R^{a_2}L^{b_2}\cdots R^{a_n}L^{b_n}, \]
where $V \in \text{SL}_2(\mathbb Z)$, $n > 0$, and $a_i$ and $b_i$ are positive integers. Moreover, the right-hand side is unique up to cyclic permutation of the factors (this fact is Proposition 2.1 in \cite{Gu_ritaud_2006}). We call $R^{a_1}L^{b_1}R^{a_2}L^{b_2}\cdots R^{a_n}L^{b_n}$ the \textbf{cyclic word} associated to $\varphi$. We will often conflate $\varphi$ with its cyclic word, letting $\varphi = R^{a_1}L^{b_1}R^{a_2}L^{b_2}\cdots R^{a_n}L^{b_n}$. Define the \textbf{size} of the cyclic word $\varphi$ to be $|\varphi| = \sum_{i=1}^n a_i+b_i$. 

We now triangulate $M_\varphi$ using the cyclic word associated to $\varphi$. The common approach in the literature is to make a rather nice connection between triangulations of tori and the Farey triangulation of $\HH^2$; see \cite{lackenby2001canonicaldecompositiononcepuncturedtorus} and \cite{Gu_ritaud_2006}. We instead take a more tangible approach for the sake of making examples.

This approach is by \textit{layering}. Here we give the rough motivation, which will be made explicit in the next paragraph. We will triangulate $M_\varphi$ such that there is a pair of faces $\Sigma$ isotopic to a fiber $S$. As $S$ travels in the circle direction, it is sheared by the map $\varphi$. Since we have a decomposition $\varphi = R^{a_1}L^{b_1}R^{a_2}L^{b_2}\cdots R^{a_n}L^{b_n}$, it is enough to ensure that as $S$ travels, it is left-sheared $b_n$ times, then right-sheared $a_n$ times, then left-sheared $b_{n-1}$ times, and so on. In this way, the fiber will realize the monodromy $\varphi$. Now, suppose the two faces of $\Sigma$ share the same tetrahedron, $A$. Then $A$ can be ``flattened" so that $\Sigma$ comprises the bottom two faces. Note that $\Sigma$ is then isotopic to the top two faces of $A$. Suppose the top two faces of $A$ are both glued to a tetrahedron $B$. Then as $\Sigma$ is isotoped into $B$, the gluings between the top faces of $A$ and the faces of $B$ will induce some action on $\Sigma$. Our goal will be to find gluings which perform a left shear ($L$) on $\Sigma$ and gluings which perform a right shear ($R$) on $\Sigma$. Then, layering the tetrahedra in a (circular) stack so that gluings perform $L$'s and $R$'s according to the decomposition of $\varphi$ will give the desired triangulation of $M_\varphi$.

We now make explicit the gluings which perform an $R$ and the gluings which perform an $L$. To make an $R$, start by immersing a tetrahedron $A$ into $M_\varphi$ such that two of its faces lie along a fiber and its other two faces lie above. Label the vertices of $A$ as in Figure \ref{fig:R Transvection}. This immersion identifies the edge connecting $3$ and $2$ to the edge connecting $0$ and $1$, and identifies the edge connecting $0$ and $3$ to the edge connecting $1$ and $2$. Now, we take the ``flattened" tetrahedron $A$ in the universal cover $\mathbb R^2$ of $T^2$, as in the top left of Figure \ref{fig:R Transvection}. Then we shear the plane by the map $R$, which descends to a map on $A$. Call the image of this map $B$, and label its vertices as in the top right of Figure \ref{fig:R Transvection}. Now stacking $B$ on top of $A$ as in the bottom of Figure \ref{fig:R Transvection} glues face $A(012)$ to face $B(013)$ and glues face $A(023)$ to face $B(123)$. This causes the bottom faces of $A$ to be isotopic to the sheared, top faces of $B$. In other words, the gluings cause the fiber $S$ isotopic to the bottom two faces of $A$ to right-shear as $S$ travels upwards into $B$.

A version of Figure \ref{fig:R Transvection} can be constructed similarly for $L$, and induces face gluings $A(012) = B(321)$ and $A(023) = B(013)$. 

These gluings occur locally between two tetrahedra. To get the global picture, we start with $|\varphi|$ tetrahedra and glue them according to the decomposition of $\varphi$ and the gluing rules above, ensuring that the final tetrahedron is glued to the first tetrahedron. A schematic of this process is given in Figure \ref{fig:L4R6 identification diagram} and explained in the next section.

\begin{figure}
\centering
    \includegraphics[width=0.6\textwidth]{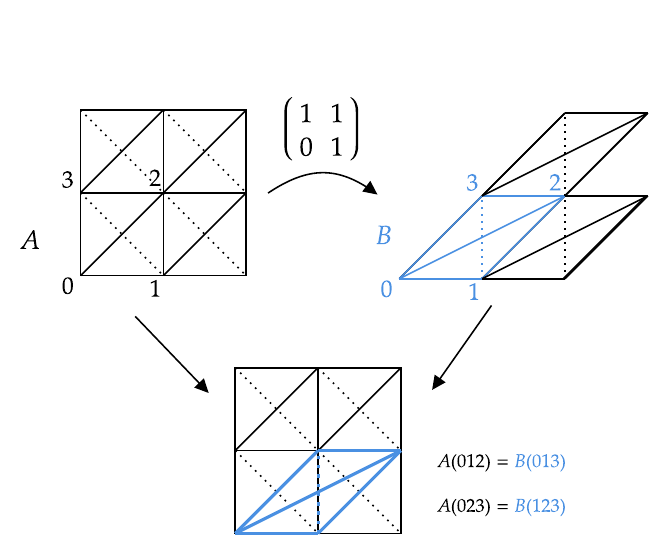}
         \caption{The layering of a right shear: taking $A$ and $B$ as tetrahedra with boundary a pleated torus, we can view each in the universal cover $\mathbb R^2$ of the tori. Right-shearing the plane corresponding to $B$ and stacking it on the plane corresponding to $A$ produces the combinatorics shown.}
         \label{fig:R Transvection}
\end{figure}

% \begin{figure}
%      \centering
%      \begin{subfigure}[b]{0.6\textwidth}
%          \centering
%          \includesvg[width=\textwidth]{figures/LTransvectionLayering}
%          \caption{The layering of an $L$: taking $A$ and $B$ as pleated tori, we can view each in the universal cover $\mathbb R^2$ of the tori. Left-shearing the plane corresponding to $B$ and stacking it on the plane corresponding to $A$ produces the combinatorics shown.
%          }
%          \label{fig:L Transvection}
%      \end{subfigure}
%      % \hfill
%      % \hspace{20px}
%      \begin{subfigure}[b]{0.6\textwidth}
%          \centering
%          \includesvg[width=\textwidth]{figures/RTransvectionLayering}
%          \caption{The layering of an $R$: taking $A$ and $B$ as pleated tori, we can view each in the universal cover $\mathbb R^2$ of the tori. Right-shearing the plane corresponding to $B$ and stacking it on the plane corresponding to $A$ produces the combinatorics shown.}
%          \label{fig:R Transvection}
%      \end{subfigure}
%      \hfill
%         \caption{}
% \end{figure}

\subsection{Example: $R^6L^4$}\label{subsection: example}
As an example, we will triangulate the once-punctured torus bundle associated to the word $\varphi = R^6L^4$. This will also serve as an example of the main theorem in the next section. Because $|\varphi|=10$, we begin with ten tetrahedra labeled $A$ through $J$. Gluing along the identifications of the previous section, we get \mbox{Figure~\ref{fig:L4R6 identification diagram}}. The resulting combinatorics is given in Table 1. We give the induced cusp triangulation in Figure~\ref{fig:L4R6 cusp}.
\begin{figure}
% \begin{subfigure}[b]{0.5\textwidth}
    \centering
    \captionsetup{width=.8\linewidth}
\includegraphics[width=0.3\linewidth]{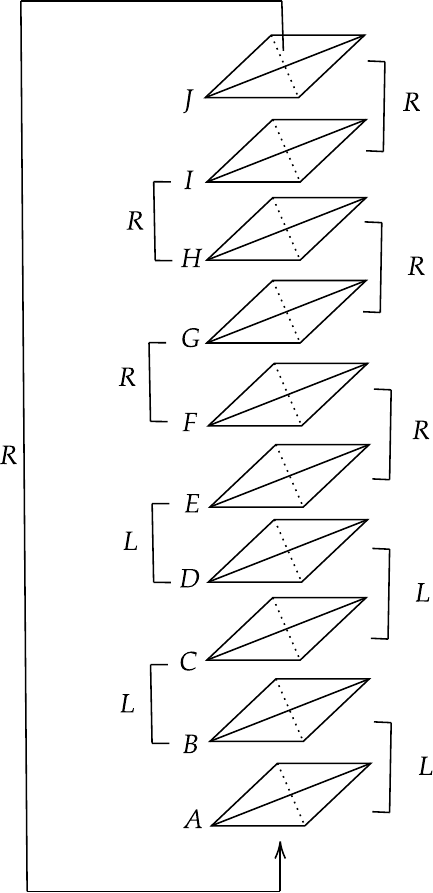}
    \caption{The monodromy ideal triangulation of the once-punctured torus bundle associated to the cyclic word $\varphi = R^6L^4$.}
    \label{fig:L4R6 identification diagram}

\end{figure}
\begin{figure}
    \centering
\captionsetup{width=.8\linewidth}
\includegraphics[width=1\linewidth]{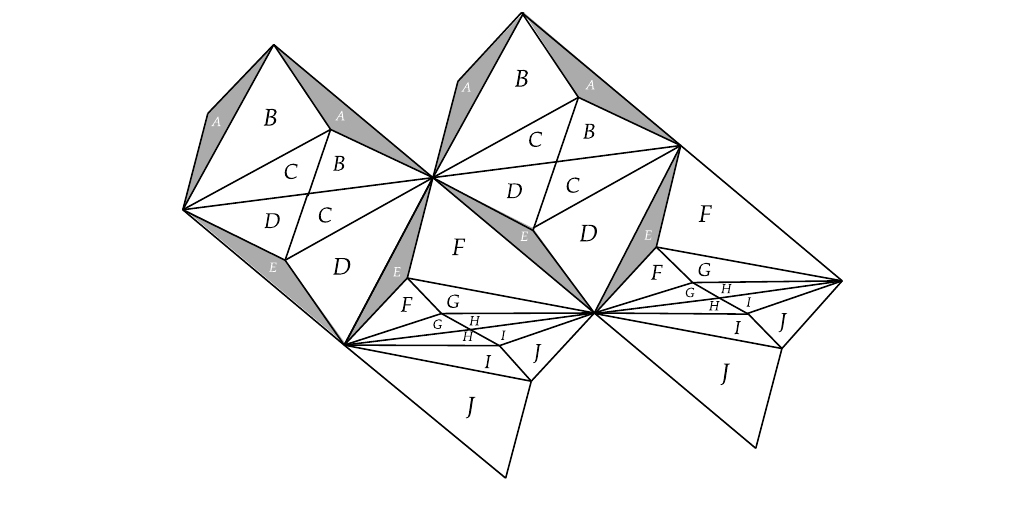}
    \caption{A fundamental domain of the cusp triangulation of the once-punctured torus bundle associated with $\varphi = R^6L^4$. The toggles are shaded in gray.}
    \label{fig:L4R6 cusp}
\end{figure}
\begin{table}
\begin{center}
    \captionsetup{width=.8\linewidth}

\begin{tabular}{|l|l|l|l|l|}
\hline
% \rowcolor{lightgray}
Tetrahedron & Face $012$ & Face $013$ & Face $023$ & Face $123$ \\ \hline
$A$         & $B(312)$   & $J(012)$   & $B(013)$   & $J(023)$   \\ \hline
$B$         & $C(312)$   & $A(023)$   & $C(013)$   & $A(120)$   \\ \hline
$C$         & $D(312)$   & $B(023)$   & $D(013)$   & $B(120)$   \\ \hline
$D$         & $E(312)$   & $C(023)$   & $E(013)$   & $C(120)$   \\ \hline
$E$         & $F(013)$   & $D(023)$   & $F(123)$   & $D(120)$   \\ \hline
$F$         & $G(013)$   & $E(012)$   & $G(123)$   & $E(023)$   \\ \hline
$G$         & $H(013)$   & $F(012)$   & $H(123)$   & $F(023)$   \\ \hline
$H$         & $I(013)$   & $G(012)$   & $I(123)$   & $G(023)$   \\ \hline
$I$         & $J(013)$   & $H(012)$   & $J(123)$   & $H(023)$   \\ \hline
$J$         & $A(013)$   & $I(012)$   & $A(123)$   & $I(023)$   \\ \hline
\end{tabular}
% \end{table}
   \caption{Gluing information for the monodromy ideal triangulation of the once-punctured torus bundle associated to the cyclic word $\varphi=R^6L^4$.}
   \end{center}
    \label{table:L4R6 identifications}
\end{table}
%%% FIGURES AND TABLE HERE
\subsection{Anatomy of the Cusp}
Our proof of the main theorem uses an understanding of the triangulation of the cusp in order to understand the triangulation of the manifold. In the case of monodromy ideal triangulations, the layering structure lends itself to an especially nice cusp triangulation. Let $T$ be the monodromy ideal triangulation of $M_\varphi$.

The construction in Section \ref{subsection: geometric triangulations and local moves} gives a natural indexing of the tetrahedra, mod $|\varphi|$. Two adjacent tetrahedra $t_i$ and $t_{i+1}$ share two faces whose gluings together represent either an $L$ or an $R$. Thus, it makes sense to label the pairs of faces between tetrahedra as either $L$ or $R$; see Figure \ref{fig:L4R6 identification diagram} for an example. If $t_{i-1}$, $t_{i}$, $t_{i+1}$ is a sequence of tetrahedra and the label between $t_{i-1}$ and $t_{i}$ differs from the label between $t_{i}$ and $t_{i+1}$, we call $t_i$ a \textbf{toggle}\footnote{Some authors use the term ``hinge".}. The tetrahedra between two toggles, including the toggles themselves, collectively form a \textbf{fan}. In the example of Section \ref{subsection: example}, tetrahedra $A$ and $E$ are toggles separating the fans $\{A, B, C, D, E\}$ and $\{E, F, G, H, I, J, A\}$. 
% Due to the layering of the tetrahedra, there is a natural labeling of the faces of $T$ (each corresponding to three edges of the cusp triangulation) by `$L$' and `$R$'. If the label just before and just after a cusp triangle are different, then we call the associated tetrahedron a \textbf{toggle}\footnote{Some authors use the term hinge.}. The tetrahedra between two toggles collectively form a \textbf{fan}. In Figure \ref{fig:L4R6 cusp}, tetrahedra $A$ and $E$ are toggles separating the fans $\{B,C,D\}$ and $\{F,G,H,I,J\}$.

Now each tetrahedron of $T$ induces four triangles in the cusp triangulation. The layering of $T$ gives a layering of these triangles in the cusp triangulation, where the four triangles of $t_i$ are separated from the four triangles of $t_{i+1}$ by six edges. When the context is clear, we will use the term \textbf{fan} to refer to the collection of triangles on the cusp induced by a fan. A \textbf{half-fan} is a connected collection of triangles belonging to the same fan with exactly two triangles coming from each tetrahedron, and such that the collection is bordered by triangles coming from toggles. For example, Figure \ref{fig:fan calculations} shows two half-fans. The choice of fundamental domain in Figure \ref{fig:L4R6 cusp} makes visible the two fans and four half-fans.

There is an involution $\sigma : M_\varphi \to M_\varphi$ which acts by applying the map ${-I\in SL_2(\mathbb Z)}$ to the fibers. It is straightforward to see from the construction of $T$ that $\sigma$ fixes each tetrahedron and acts on the cusp by translation, sending each half-fan to its other half. This fact will allow us to simplify many calculations in the next section.

We end this section with the following proposition.
\begin{proposition}\label{prop: even degree}
    The degree of each edge in the triangulation $T$ is even. Equivalently, the valence of each vertex in the induced cusp triangulation is even.
\end{proposition}
\begin{proof}
    This is the seventh remark in \cite[Section~4]{Gu_ritaud_2006}. In fact, a vertex of the cusp has either valence $4$ or valence $2k+2$ where $k$ is the size of an incident fan.
\end{proof}
\section{Proof of the Main Theorem}
Let $\varphi = R^NL^M$, and let $M_\varphi$ be the associated once-punctured torus bundle. Following \cite[Section 10]{Gu_ritaud_2006}, when $N$ and $M$ are sufficiently large, there are complex constants $a$, $b$, $a'$, and $b'$ depending on $N$ and $M$ such that
\begin{equation*}
    \begin{cases}
      \sin a &= \;\;\;i\tan b \cos b'\\
      \sin a' &= -i\tan b' \cos b
    \end{cases}  \;\;\;\;\;\;\;\;\text{with}\;\;\;\;
    \begin{cases}
        b &= (\pi - 2a)/N\\
        b' &= (\pi - 2a')/M.
    \end{cases}
\end{equation*}
 Throughout this section, we assume $N$ and $M$ are sufficiently large. 
We take the following proposition from \cite{Gu_ritaud_2006}:
\begin{proposition}\cite[Proposition 10.1]{Gu_ritaud_2006}
\label{proposition: gueritaud-10.1}
    The half-fan corresponding to $R^N$ can be embedded into $\CC$ with nodes at complex coordinates $\pm\cot(b)$ and intermediary vertices $\cot(a+sb)$ where $-1 \leq s \leq N+1$; similarly, the half-fan corresponding to $L^M$ can be embedded into $\CC$ (possibly with a different scaling factor) with nodes $\pm\cot(b')$ and intermediary vertices $\cot(a' + sb')$ where $-1 \leq s \leq M+1$.
\end{proposition}
\begin{figure}
    \centering
    \captionsetup{width=.8\linewidth}
    \includegraphics[width=0.9\linewidth]{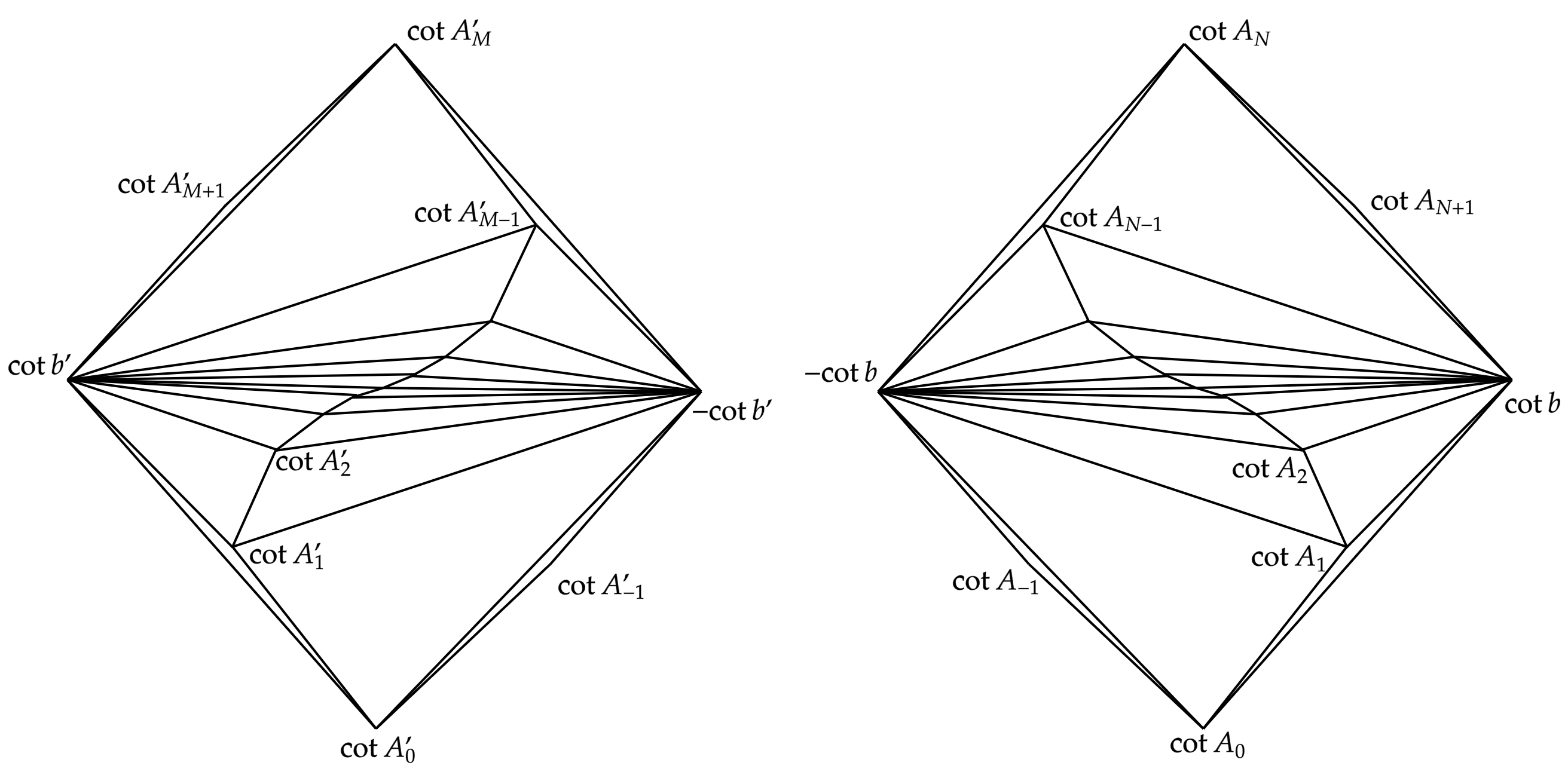} 
    \caption{The embeddings of the two half-fans of the cusp triangulation of the once-punctured torus bundle associated to $\varphi = R^N L ^M$. Adapted from \cite[Figure~11]{Gu_ritaud_2006}.}
    \label{fig:fan calculations}
\end{figure}
     See Figure \ref{fig:fan calculations}. For an integer $s$, let $A_s = a + sb$ and $A'_s = a' + sb'$. The following is a consequence of the previous proposition and the property that $\cot(z) = -\cot(\pi-z)$.
\begin{lemma}
\label{lem: involution of half-fan}
    Suppose $N$ is even. Then the half-fan corresponding to $R^N$ has a rotational involution given by the map $z \mapsto -z$ which fixes the vertex $\cot(A_{N/2})$. An analogous statement is true for the half-fan corresponding to $L^M$.
\end{lemma}
\begin{proof}
    Recall that $b = (\pi - 2a)/N$. We will first show that $\cot(A_s) = -\cot(A_{N-s})$. So for any $-1 \leq s \leq N + 1$, we have
    \begin{align*}
        \cot(a + (N-s)b) &= \cot((a + Nb) - sb)\\
        &= \cot((a + \pi - 2a) - sb)\\
        &= \cot((\pi - (a + sb))\\
        &= -\cot(a + sb).
    \end{align*}
    It follows the map $z\mapsto -z$ is a symmetry of the half-fan. Since $N$ is even, there is a vertex at $\cot(a + (N/2)b) = \cot(a + (\pi - 2a)/2) = \cot(\pi/2) = 0$, which is fixed.  Replacing $a$, $b$, and $N$ with $a'$, $b'$, and $M$ appropriately gives the statement for the half-fan corresponding to $L^M$.
\end{proof}
Finally, we'll need the following geometrical lemma.
\begin{lemma}
\label{lem: geometrical}
    Suppose $N$ is even. Suppose we have four consecutive points $P_j = \cot(a + (s+j)b)$ for some fixed $0\leq s \leq N-3$ and $j\in\{0,1,2,3\}$. Let $Q_1 = -\cot(b)$ and let $Q_2 = \cot(b)$. Let $\angle012$ be the angle formed by $P_0$, $P_1$, and $P_2$ facing $Q_1$, and let $\angle123$ be the angle incident to $P_1$, $P_2$, and $P_3$ facing $Q_2$, valued in $[0,2\pi)$. Then exactly one of the following is true:
    \begin{itemize}
        \item $\angle012 < \pi$ and $\angle123 > \pi$;
        \item $\angle012 < \pi$ and $\angle123 = \pi$;
        \item $\angle012 = \pi$ and $\angle123 < \pi$; or
        \item $\angle012 > \pi$ and $\angle123 < \pi$.
    \end{itemize}
\end{lemma}
\begin{figure}
        \centering
        \includegraphics[width=0.8\linewidth]{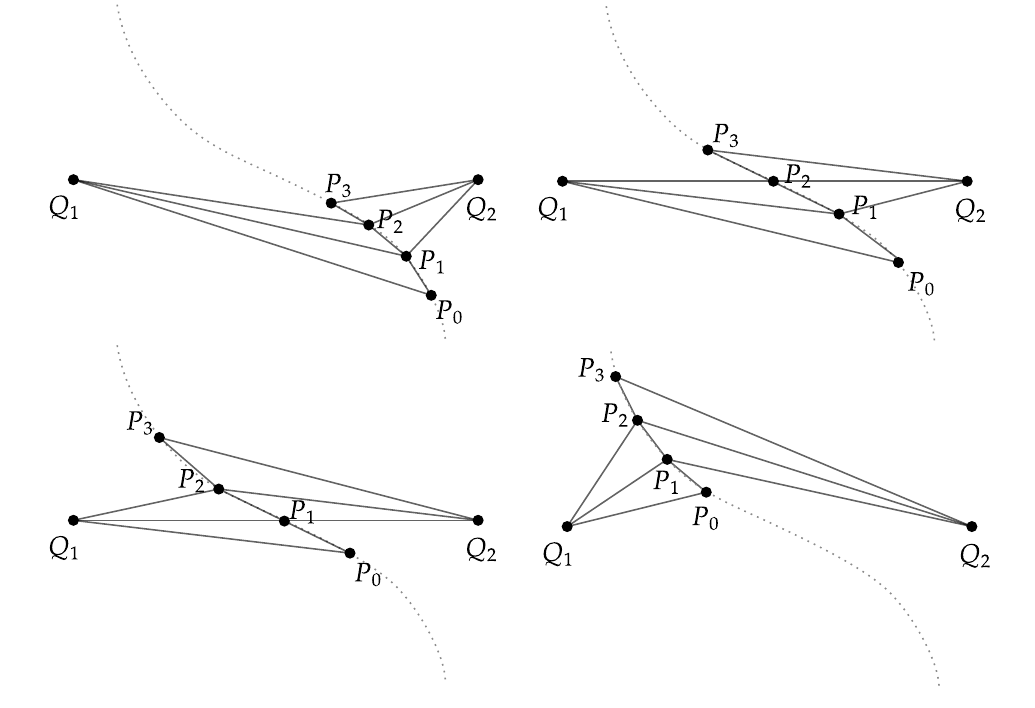}
         \captionsetup{width=.8\linewidth}
        \caption{The four cases of Lemma \ref{lem: geometrical}. The dotted line is part of the curve $\cot(a + tb)$.}
        \label{fig:four cases}
    \end{figure}\begin{proof}
    Figure \ref{fig:four cases} gives the idea of the proof, but we will compute the angles explicitly. Recall the elementary formula $\cot(z_1) - \cot(z_2) = \frac{\sin(z_2-z_1)}{\sin(z_1)\sin(z_2)}$. In what follows, we choose $\arg$ to take values in $[0, 2\pi)$, branched at $\mathbb R_{\geq 0}$. The angle $\angle012$ is given by
    \begin{align*}
        \arg\left(\frac{P_0-P_1}{P_2-P_1}\right) &= \arg\left(\frac{\cot(A_{s})-\cot(A_{s+1})}{\cot(A_{s+2})-\cot(A_{s+1})}\right)\\
        &= \arg\left(\frac{\sin(A_{s+1}-A_{s})\sin(A_{s+2})\sin(A_{s+1})}{\sin(A_{s})\sin(A_{s+1})\sin(A_{s+1} - A_{s+2})}\right)\\
        &= \arg(-\sin(A_{s+2})) - \arg(\sin(A_{s}))\\
        &= \pi + \arg(\sin(A_{s+2})) - \arg(\sin(A_{s})).
    \end{align*}
    To calculate $\angle123$, we rotate all points by $\pi$ about $P_2$ to avoid the branch cut. A similar calculation gives $\angle123 = \pi + \arg(\sin(A_{s+1})) -\arg(\sin(A_{s+3}))$. For notational convenience, let $f(s) := \arg(\sin(A_s))$. Appealing to Figure \ref{fig:contour plot}, $f(s) \in [0, \pi/2)$ and $f(N/2) = 0$. Furthermore, $f(s)$ decreases from $s =0,\dots, N/2$ and increases from $s = N/2,\dots,N$ (a quick check from the definitions shows that $N/2$ is the only critical point of $f$ in $[0,N]$). Taking angles mod $2\pi$ appropriately, it follows that 
    \begin{itemize}
        \item if $f(s+2) - f(s) < 0$ then $\angle012 \in (0,\pi)$,
        \item if $f(s+2) - f(s) = 0$ then $\angle012 = \pi$, and
        \item if $f(s+2) - f(s) > 0$ then $\angle012 \in (\pi,2\pi)$.
    \end{itemize}
    The similar statement for $\angle123$ is given by considering the sign of $f(s+1) - f(s+3)$. There are four cases:
    \begin{itemize}[-]
        \item Case 1: $s \leq (N/2) - 3$. In this case, $f(s+2)-f(s) < 0$ while $f(s+1)-f(s+3)>0$. So $\angle012 < \pi$ and $\angle123 > \pi$.
        \item Case 2: $s = (N/2) - 2$. In this case, $f(s+2)-f(s) < 0$ while $f(s+1)=f(s+3)$. So $\angle012 < \pi$ and $\angle123 = \pi$.
        \item Case 3: $s = (N/2) - 1$. In this case, $f(s+2) = f(s)$ while $f(s+1)-f(s+3)<0$. So $\angle012 = \pi$ and $\angle123 < \pi$.
        \item Case 4: $s \geq (N/2)$. In this case, $f(s+2)-f(s)>0$ while $f(s+1)-f(s+3) < 0$. So $\angle012 > \pi$ and $\angle123 < \pi$.
    \end{itemize}
\end{proof}
\begin{figure}
        \centering
        \includegraphics[width=0.7\linewidth]{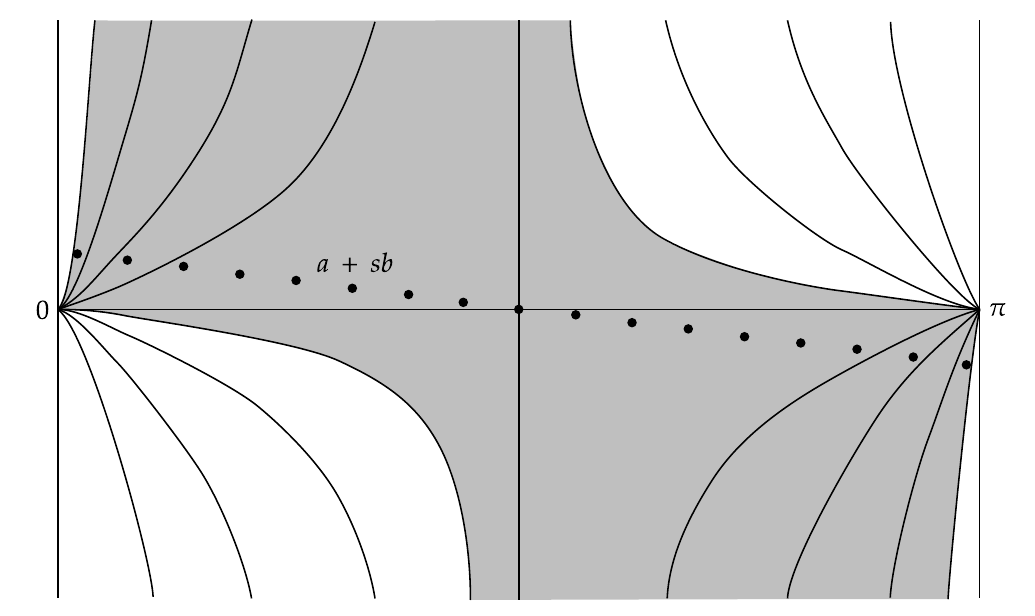}
         \captionsetup{width=.8\linewidth}
        \caption{Level curves of $z \mapsto \arg(\sin(z))$ for $0 \leq \text{Re}(z)\leq \pi$. The shaded region is the ``authorized region" discussed in \cite[Section 10]{Gu_ritaud_2006}. In particular, $\arg(\sin(a + sb)) \in [0,\pi/2)$ for $0\leq s \leq N$.}
        \label{fig:contour plot}
    \end{figure}
We are now ready to prove the main result.
\begin{theorem}
\label{theorem: main}
    The once-punctured torus bundle associated to the cyclic word $R^{N}L^{M}$ has an isolated geometric triangulation for all $N$ and $M$ even and sufficiently large.
\end{theorem}
\begin{proof}
    Let $N$ and $M$ be even and let $\varphi = R^{N}L^{M}$. Let $T$ be the monodromy ideal triangulation of $M_\varphi$. We want to show that there are no 2-3 or 3-2 moves which can be applied to $T$ and which result in a triangulation that is geometric.
    
    Firstly, there are no 3-2 moves, since $T$ has no degree-3 edges by Proposition~\ref{prop: even degree}.
    
    Now, any 2-3 move must involve two tetrahedra in the same fan (if one of the tetrahedra is a toggle, then it belongs to both fans). Thus, we may restrict our attention to the fan corresponding to $R^{N}$ without loss of generality.
    
    Recall that a 2-3 move induces three 2-2 moves and two 1-3 moves on the cusp triangulation. We distinguish two cases:
    \begin{itemize}[-]
    \item Case 1: Neither tetrahedron is a toggle.
   Each tetrahedron induces four triangles in the cusp fan corresponding to $R^N$. See Figure \ref{fig:four cases} for the possible cases in a half-fan. Performing a 2-3 move between the two tetrahedra induces three 2-2 moves between six of the eight induced cusp triangles. We first show that we can pick a half-fan for which the 2-3 move induces 2-2 moves that flip (in the language of Lemma \ref{lem: geometrical}) both $Q_1P_1$ and $Q_2P_2$.

   Fixing a half-fan, we see that the only possible 2-2 moves between the relevant triangles flip the edges $Q_1P_1$, $Q_2P_2$, $P_1P_2$, $\sigma(Q_1P_1)$, $\sigma(Q_2P_2)$, or $\sigma(P_1P_2)$. Note that if $P_1P_2$ is flipped, then $\sigma(P_1P_2)$ cannot be flipped, as there would be no remaining adjacent triangles for the third induced 2-2 move. It follows, then, that either $Q_1P_1$ and $Q_2P_2$ are both flipped or $\sigma(Q_1P_1)$ and $\sigma(Q_2P_2)$ are both flipped. Since the half-fans are isometric, we will say without loss of generality that the 2-3 move causes the edges $Q_1P_1$ and $Q_2P_2$ to be flipped.
    
    Now if the 2-3 move is geometric, then the angles $\angle012$ and $\angle123$ must both be less than $\pi$. But by Lemma \ref{lem: geometrical}, this isn't possible.
    \item Case 2: One tetrahedron is a toggle. Using the labels in Figure \ref{fig:toggle case}, we have that a 2-2 move induced by a 2-3 move must swap the diagonal along the lines $Q_1P_1$ or $Q_2P_2$. So it suffices to show that one of these 2-2 moves cannot be geometric. We will show that $\angle P_0P_1P_2 := \angle012 > \pi$. As in the proof of Lemma \ref{lem: geometrical}, we have $\angle012= \pi + \arg(\sin(A_{N})) - \arg(\sin(A_{N-2}))$. Now appealing to Figure \ref{fig:contour plot} once again, we get ${\arg(\sin(A_{N})) - \arg(\sin(A_{N-2})) > 0}$, and thus $\angle012 >\pi$.
    \end{itemize}
    In either case, a geometric 2-3 move is not possible.
\end{proof}
    \begin{figure}
        \centering
        \includegraphics[width=\linewidth]{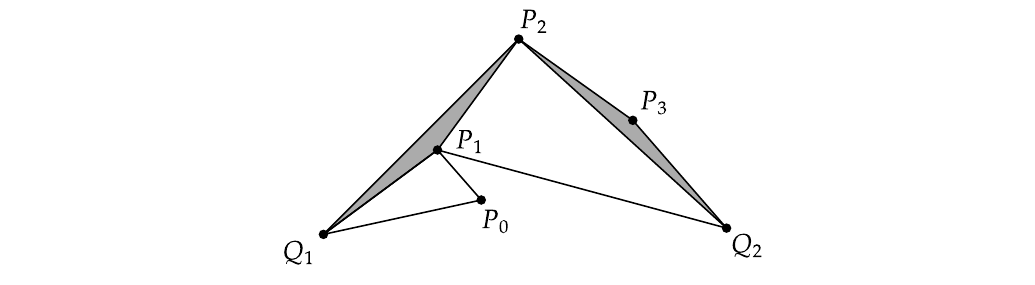}
         \captionsetup{width=.8\linewidth}
        \caption{The case where one of the tetrahedra involved in the 2-3 move is a toggle. The toggle is shaded in gray.}
        \label{fig:toggle case}
    \end{figure}
    % Observe that if $\varphi = L^MR^N$, there is a geometric 2-3 move if either $M$ or $N$ is odd; this happens between the pink and green tetrahedra in Figure \ref{fig:around inflection odd}. In the case where $M$ and $N$ are both even, the corresponding 2-3 move induces a flat tetrahedron.

\section{Other Geometric Triangulations of Once-Punctured Torus Bundles}
Let $\varphi = R^N L^M$ with $M$ and $N$ sufficiently large and both even, and let $T$ be the monodromy ideal triangulation of $M_\varphi$. By Theorem \ref{theorem: main}, there are no 2-3 or 3-2 moves that can be applied to $T$ to produce a different geometric triangulation. The goal of this section is to show that other geometric triangulations of $M_\varphi$ exist.

\begin{lemma}
\label{lemma: move sequence}
    Starting from the monodromy ideal triangulation $T$ of $M_\varphi$, there is a sequence of two 2-3 moves followed by a 3-2 move resulting in a geometric triangulation of $M_\varphi$ which is not isomorphic to $T$.
\end{lemma}
\begin{proof}
    The layering of tetrahedra in $T$ induces a natural indexing of the tetrahedra in each fan; label the tetrahedra in the fan corresponding to $R^{N}$ as $t_0,\dots,t_{N+1}$.  The sequence of moves we wish to perform involves only tetrahedra $t_{\frac N2-1}$, $t_{\frac N2}$, and $t_{\frac N2+1}$ (which exist since $N$ is sufficiently large).

    Note that a 2-3 move induces three 2-2 moves and two 1-3 moves on the cusp, while a 3-2 move induces three 2-2 moves and two 3-1 moves (Figure \ref{fig:2-3 ideal cusp}). We illustrate the sequence of moves on the cusp in Figure \ref{fig:regeometrizing sequence on the cusp}; these correspond to 2-3 and 3-2 moves in the triangulations of $M_\varphi$. Because the sequence of moves affects only these three tetrahedra, the shapes of all other tetrahedra remain the same, and so Figure \ref{fig:regeometrizing sequence on the cusp} is a sufficient characterization of the resulting triangulation.

    \begin{figure}
        \centering
        \includegraphics[width=\linewidth]{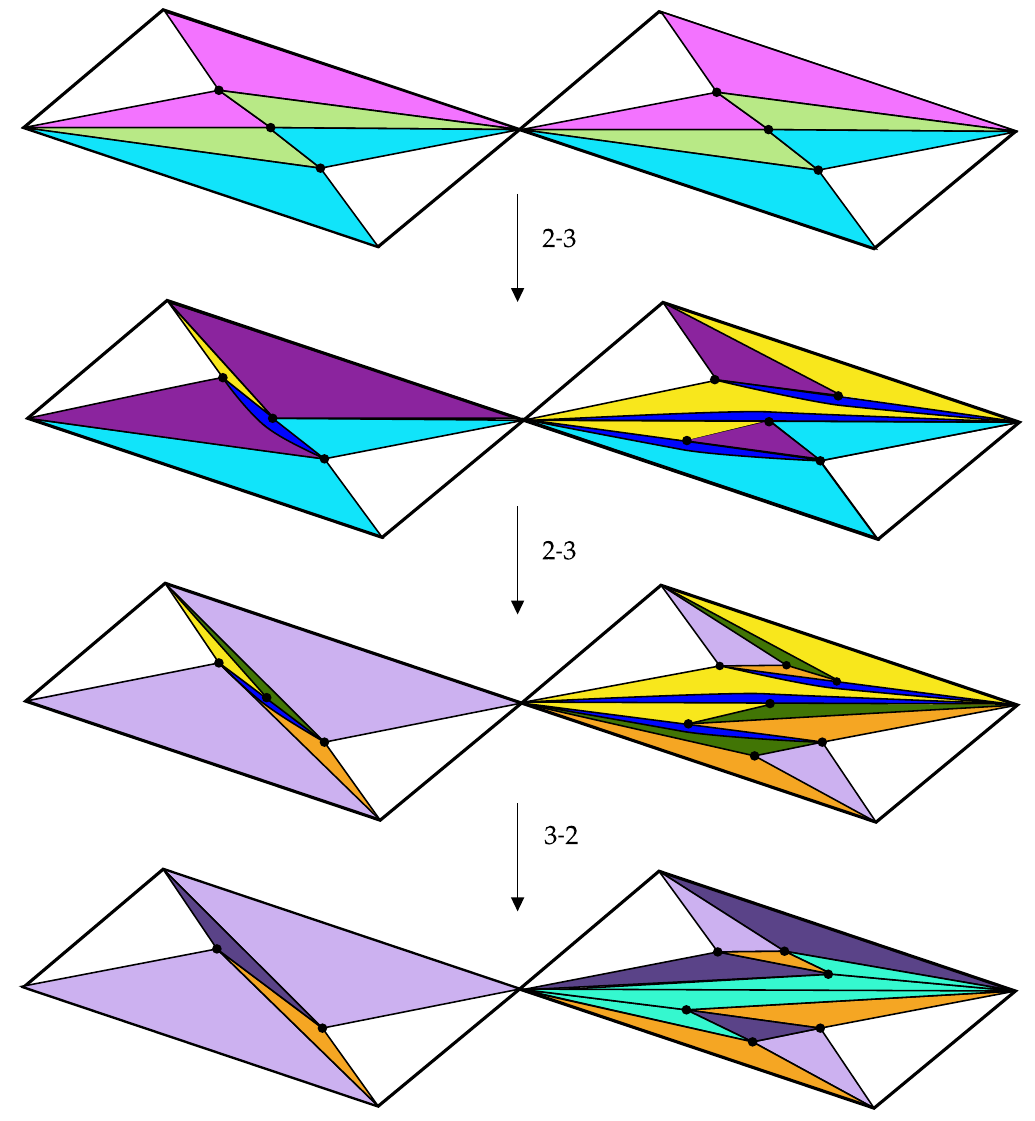}
         \captionsetup{width=0.8\linewidth}
        \caption{The effect of the sequence of moves from Lemma \ref{lemma: move sequence} on the relevant portion of the cusp. Each color represents a tetrahedron. The first triangulation is the monodromy ideal triangulation $T$; shown are $t_{\frac N2+1}$ (pink, upper layer), $t_{\frac N2}$ (green, middle layer), and $t_{\frac N2-1}$ (blue, lower layer). The dark blue triangles are flat triangles which have been blown up for visibility.}
        \label{fig:regeometrizing sequence on the cusp}
    \end{figure}
    
    The first move in the sequence is a 2-3 move between $t_{\frac N2+1}$ and $t_\frac N2$. This results in three tetrahedra, say $t_p$ (purple), $t_y$ (yellow), and $t_f$ (flat, dark blue). The tetrahedron $t_f$ is flat, as the move instantiates Case 3 of Lemma \ref{lem: geometrical}.
    
    The second move is a 2-3 move between $t_{\frac N2-1}$ and $t_p$. This move introduces no flat or negatively-oriented tetrahedra, as can be seen in Figure \ref{fig:regeometrizing sequence on the cusp} and can be verified in the following way: the two 2-2 moves in the left half-fan are geometric, which can be checked by showing that $f(\frac N2 +1)-f(\frac N2-2) < 0$ and $f(\frac N2 -1)-f(\frac N2 +2) < 0$ (by using the fact that $f(\frac N2 + s)=f(\frac N2-s)$), as in the proof of Lemma \ref{lem: geometrical}. The 2-2 move in the right half-fan being geometric follows from $t_y$ being positively-oriented (in particular, not flat). Call the new tetrahedra $t_g$ (green), $t_{lp}$ (light purple), and $t_o$ (orange).

    The final move is a 3-2 move which replaces $t_f$, $t_y$, and $t_g$ with two new tetrahedra (dark purple and cyan). Suppose $e$ is the shared degree-3 edge. To show that the result of performing a 3-2 move dual to $e$ is geometric, we refer to Figure \ref{fig:2-3 with shapes}. Using the notation in the proof of Proposition $\ref{prop: isolated not geom}$, suppose $t_y$ has shape parameter $r$, $t_g$ has shape parameter $u$, and $t_f$ has shape parameter $v$. Observe that $\arg(v') = 0$. Now $\arg(z')=\arg(u'')+\arg(v') = \arg(u'')$ and $\arg(w) = \arg(r'')+\arg(v') = \arg(r'')$. Since $t_y$ and $t_g$ are positively oriented, $\arg(u''),\arg(r'') \in(0,\pi)$, and thus the resulting two tetrahedra are also positively oriented.

    The new geometric triangulation $T'$ has $|\varphi|+1$ tetrahedra, so $T'$ is not isomorphic to $T$.
\end{proof}

% \begin{remark}
%     The proof above performs these moves on the cusp corresponding to $R^{N}$, but the same can be done for the cusp corresponding to $L^{M}$. When $M\neq N$, this gives another geometric triangulation. Computer searches using the programs SnapPy \cite{SnapPy} and Regina \cite{regina} on small values of $N$ and $M$ suggest that many more geometric triangulations exist.
% \end{remark}

This shows that $M_\varphi$ has at least two geometric triangulations. As defined in the introduction, the geometric bistellar flip graph $\mathbb{T}_G(M_\varphi)$ is the graph whose nodes are the geometric ideal triangulations of $M_\varphi$ and whose arcs connect nodes if and only if those triangulations differ by a 2-3 move. If $N > N'$, then the manifolds $R^{N}L^{N}$ and $R^{N'}L^{N'}$ are distinct as manifolds (the traces of the monodromy matrices of the torus fibering differ, for example), and so the family in consideration is indeed infinite. So, together with Theorem $\ref{theorem: main}$, we get the following:
\begin{corollary}
\label{cor:infinitely many disconnected}
There are infinitely many cusped hyperbolic 3-manifolds $M$ for which the geometric bistellar flip graph is disconnected.
\end{corollary}

\bibliographystyle{plain}
\bibliography{refs}
\end{document}